\documentclass[12pt]{article}

\usepackage{amssymb}
\usepackage{amsthm}
\usepackage{amsmath}
\usepackage{mathdots}
\usepackage{microtype}

\usepackage{graphicx}

\usepackage{float}

\newtheorem{theorem}{Theorem}[section]
\newtheorem{lemma}[theorem]{Lemma}
\newtheorem{proposition}[theorem]{Proposition}
\newtheorem{conjecture}[theorem]{Conjecture}

\theoremstyle{definition}

\newtheorem{example}[theorem]{Example}

\begin{document}

\title{A note on a conjecture concerning rank one perturbations of singular M-matrices}

\author{B.~Anehila\footnote{ROC Mondriaan, Theresiastraat 8, 2593 AN Den Haag
E-mail: \texttt{b.anehila@rocmondriaan.nl} or \texttt{b.anehila@gmail.com}},
A.C.M.~Ran\footnote{Department of Mathematics, Faculty of Science, VU Amsterdam, De Boelelaan
    1111, 1081 HV Amsterdam, The Netherlands
    and Research Focus: Pure and Applied Analytics, North-West~University,
Potchefstroom,
South Africa. E-mail:
    \texttt{a.c.m.ran@vu.nl}}}

\date{}

\maketitle

\begin{abstract}
A conjecture  from \cite{BR} concerning the location of eigenvalues of rank one perturbations of singular M-matrices is shown to be false in dimension four and higher, but true for dimension two, as well as for dimension three with an additional condition on the perturbation.
\end{abstract}



\section{Introduction}

Let $H$ be a nonnegative $n \times n$ matrix, and let $A=\rho(H)I-H$. Here, $\rho(H)$ denotes the spectral radius of $H$. Note that $A$ has all its eigenvalues in the open right half plane, with the exception of zero. We assume that $\rho(H)$ is a simple eigenvalue of $H$, i.e., both the geometric and the algebraic mutliplicity are one. Then zero is a simple eigenvalue of $A$.
Further, let $v$ and $w$ be nonnegative vectors. 
For $t>0$ consider the matrix $B(t)=A+tvw^\top $. It was shown in 
\cite{BR}, Lemma 2.11, that there is a $t_0>0$ such that for $0<t<t_0$ the matrix $B(t)$ has all its eigenvalues in the open right half plane under a certain extra condition, the so-called NZP condition, which will be introduced shortly. It was also shown in \cite{BR} by means of a counterexample (see Example 2.15) that this does not hold for all $t>0$. The counterexample in \cite{BR} is of size $6\times 6$. 
 
In order to state one of the conjectures in \cite{BR}, which is the focus of this note, we need to introduce a condition called NZP (non-zero projection) in \cite{BR}. For this, consider a right unit eigenvector $z_r$ of $H$ corresponding to $\rho(H)$ and a left unit eigenvector $z_l$ of $H$ corresponding to $\rho(H)$.  Note that $z_l$ and $z_r$ are unique up to multiplication by $-1$. 
The vectors $v$ and $w$ are said to satisfy NZP if the following hold: $z_l^\top  v\not=0$ and $w^\top  z_r\not=0$.

If $H$ is irreducible, then $z_r$ and $z_l$ have positive entries (see, e.g., \cite{BP}, Chapter 2, Theorem 2.10, part c). Hence, if $v$ and $w$ are nonnegative vectors and $H$ is irreducible, then the condition NZP is automatically satisfied.

The following conjecture was stated in \cite{BR}, Conjecture 2.17.

\begin{conjecture}
Let $H$ be a nonnegative $n\times n$ matrix, let $A=\rho(H)I-H$. Assume that $0$ is a simple eigenvalue of $A$. Let $v$ and $w$ be nonnegative vectors satisfying NZP, then there
is a positive $t_1$ such that for $t>t_1$ the eigenvalues of $B(t)=A+tvw^\top $ are again all in the open right half plane. 
\end{conjecture}

In \cite{BR}, Theorem 2.7 the conjecture was already shown to be true in the two-dimensional case assuming that either the zero eigenvalue of $A$ is simple, or $w^\top v \not= 0$. In fact, in these cases the eigenvalues of $B(t)$ are both in the open right half plane for all $t>0$. 
The purpose of this short note is to show that the conjecture stated above is false in general, even when we make the extra assumption that $H$ is irreducible,
but true in the three-dimensional case under the assumtion that $H$ is irreducible and $w^\top v\not=0$. It will be shown also that the latter condition is necessary.

To put the conjecture in context, the problem of studying the behaviour of eigenvalues of parametrized rank one perturbations $B(t)=A+tvw^\top$ of a matrix $A$ has been studied for a long time, see e.g., \cite{Baumgartel,Kato}. Mostly, however, only the behaviour for small values of $t$ has been studied (\cite{Lidskii, VL}, see also \cite{Karow, Moro-Burke-Overton} for more detailed analysis). The problem of considering the behaviour of the eigenvalues for large values of $t$ was considered in \cite{RW1,RW2}. 
Restrictions on $A$, $v$ and $w$, allowing only certain structured matrices, where considered in e.g. \cite{MMRR1}, but only for generic vectors $v$ and $w$. For the non-structured situation, without restrictions on $A$, $v$ and $w$, and generic vectors $v$ and $w$, see \cite{HM,MD,Sa1,Sa2}.
In \cite{BR} the matrix $A$ and the vectors $v$ and $w$ are restricted in a different manner: $A$ is a singular $M$-matrix, and $v$ and $w$ are nonnegative vectors. In \cite{BR}, Lemma 2.11 it was shown that for small values of the parameter $t>0$ the eigenvalues of $B(t)$ are all in the open right half plane when the condition NZP is satisfied (in particular when $H$ is irreducible).
The conjecture above asks explicitly for the behaviour of the eigenvalues of $B(t)$ for large values of $t>0$.

An important role in our arguments is played by the following formula for the characteristic polynomial of $B(t)$: let $m_A(\lambda)$ denote the minimal polynomial of $A$, then
$$
\det(\lambda I_n-B(t))= \frac{\det(\lambda I_n -A)}{m_A(\lambda)}(m_A(\lambda)-tp_{vw}(\lambda)),
$$
where $p_{vw}(\lambda)=m_A(\lambda)w^\top  (\lambda I_n -A)^{-1}v$. Note that formally the equation only holds for $\lambda$ not one of the eigenvalues of $A$, but since both sides are polynomials, they then coincide everywhere.

\section{Counterexample and general remarks}

We start by providing a counterexample in dimension four. Let
$$
H=\begin{bmatrix}
0.1  &  1   &    0  &   0\\
0   & 0.1   & 1 &    0\\
0     &    0 &   0.1 &   1\\
10^{-4}      &   0     &    0  &  0.1\\
\end{bmatrix}, \quad
v=
\begin{bmatrix}
2\\
0.1\\
0.1\\
2\\
\end{bmatrix}, \quad
w=\begin{bmatrix}
2\\
0.1\\
2\\
0.1\\
\end{bmatrix}.
$$
Then $H$ is nonnegative and irreducible, and the eigenvalues of $H$ are $0, 0.1\pm 0.1 i, 0.2$ and so $\rho(H)=0.2$. Consider
$A=\rho(H)I_4-H$, and $B(t)=A+t vw^\top $. Note that $H$, $v$ and $w$ satisfy all the conditions of the conjecture.
Following \cite{RW2}, Theorem 17, see also \cite{RW1}, Theorem 4.1 we have
that for $t\to\infty$ the eigenvalues of $B(t)$ behave as follows: one is  positive, and approximately equal to $t w^\top  v+O(1)$, and the other three converge to the roots of the polynomial $p_{vw}(\lambda)=w^\top  m_A(\lambda)(\lambda I_4-A)^{-1}v$, where $m_A(\lambda)$ is the minimal polynomial of $A$. In this case, $p_{vw}(\lambda)$ is given by
$$
p_{vw}(\lambda)=\det(\lambda I-A)-\det(\lambda I- (A+vw^\top )).
$$
For this specific case, this is equal to
$$
p_{vw}(\lambda)= 4.4100\lambda^3   - 5.5330\lambda^2 +    1.3747\lambda    - 4.0866,
$$
which has roots in $1.4710$ and $ - 0.1082 \pm 0.7863 i$. Hence for large values of $t$ the matrix $B(t)$ has two eigenvalues in the open left half plane, and these eigenvalues do not converge to values in the open right half plane. The eigenvalues are plotted as functions of $t$ in Figure 1 below.

\begin{figure}[H]
\centering
\includegraphics[height=5.5cm,width=10cm]{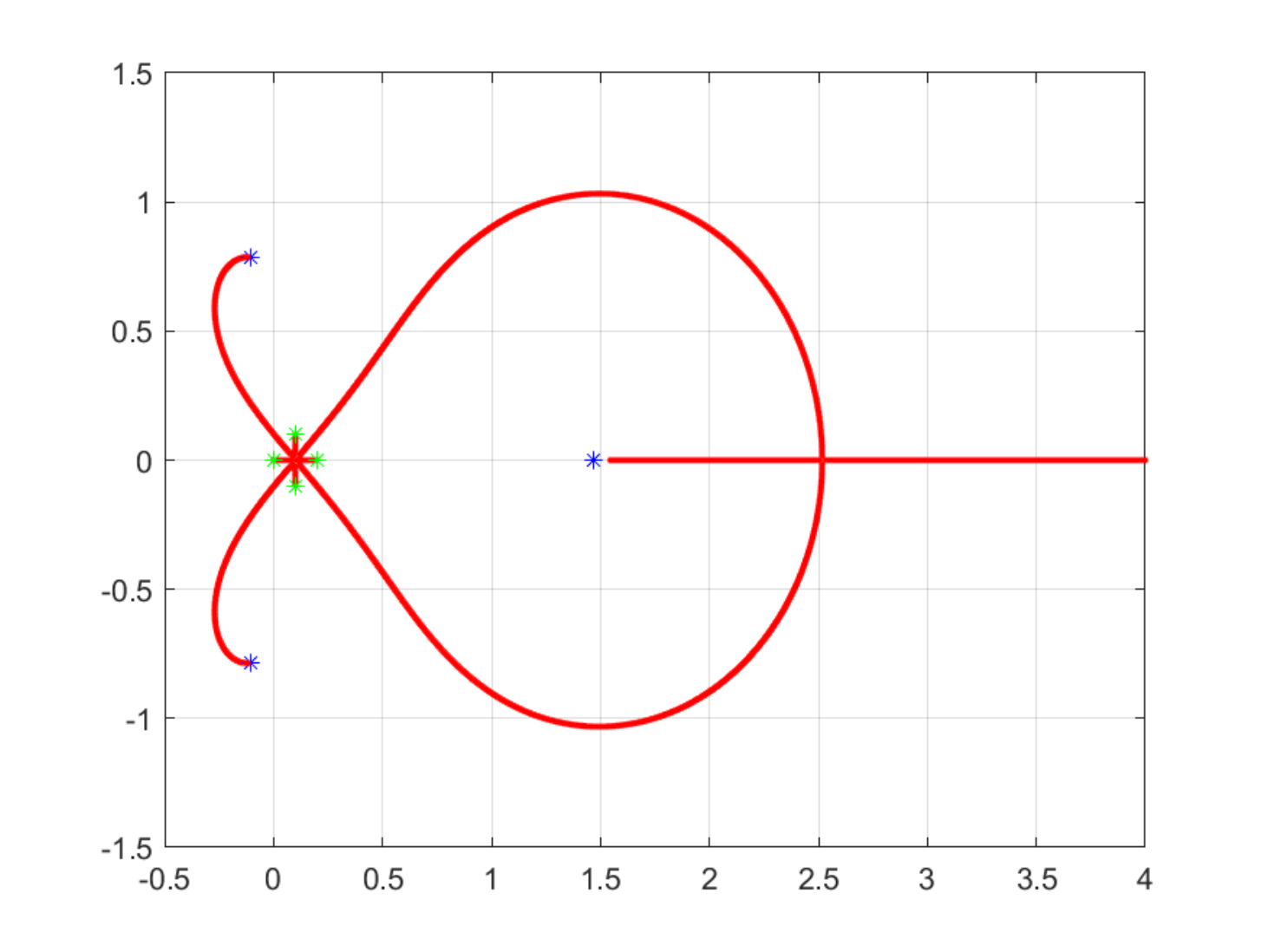}
\caption{Eigenvalue curves of $B(t)$ plotted as functions of $t$. The green stars indicate the eigenvalues of $A=\rho(H)I_4-H$, and the blue stars are the roots of $p_{vw}$.}
\end{figure}

It should be noted that the statements in \cite{RW1}, Theorem 4.1 are made for generic vectors $v$ and $w$ only. However, the fact that $v$ and $w$ have positive entries allows us to apply the results of that theorem in this particular case, which can be seen by a close inspection of the proofs in \cite{RW1}, Lemma 2.1 and Theorem 4.1. The crucial condition is that $w^\top  v\not=0$. Alternatively, this can be seen by applying \cite{RW2}, Theorem 17, which does not depend on generiticity of the vectors $v$ and $w$. ``Generic'' here is taken in the algebraic-geometric sense, that is, the set of vectors $(v,w)$ for which the stated property is not true is contained in the zero set of a finite number of polynomials in the $2n$ variables which are the coordinates of $v$ and $w$.

Obviously, once a counterexample is found in dimension four, counterexamples in any higher dimension can be constructed easily. In fact, if we denote the Jordan block with eigenvalue zero of size $n$ by $J_n$, we see that apart from the $10^{-4}$ in the $4,1$-entry $H$ is equal to $0.1 I_4+J_4$. In dimension $n$ we can construct $H$ as follows: take $0.1 I_n+J_n$ and insert in the $n,1$-entry the number $10^{-n}$. Then again $\rho(H)=0.2$. An appropriate choice of positive vectors $v$ and $w$ will lead to a counterexample in any dimension. We will not go in further detail here.

The counterexample to the conjecture leaves the question what the conditions on $A$, $v$ and $w$ are for the eigenvalues of the matrix $B(t)$ to be in the open right half plane for large values of $t$.
We assume that $n>3$, as the cases $n=2$ and $n=3$ will be discussed in the next section. We make use of \cite{RW2}, Theorem 17 and Lemma 16. In fact, from Theorem 17 (ii) and (iii) in \cite{RW2} we see that it is a necessary condition that $w^\top v\not= 0$, as otherwise there is at least one eigenvalue of $B(t)$ going to infinity in the closed left half plane. Under this condition, the eigenvalues of $B(t)$ are in the open right half plane for large values of $t$ if and only if the roots of the polynomial $p_{vw}(\lambda)=m_A(\lambda)w^\top(\lambda I_n-A)^{-1}v$ are in the open right half plane. Lemma 16 in \cite{RW2} gives a formula for the coefficients of this polynomial: if 
$m_A(\lambda)=\sum_{k=0}^l m_k\lambda^k$, then 
$$
p_{vw}(\lambda)=\sum_{i=0}^{l-1} \left(\sum_{\substack {k-j=i+1\\ k,j\geq 0}}m_k w^\top A^{j} v\right) \lambda^i .
$$
One can then apply the Routh-Hurwitz criterion to this polynomial to see whether its roots are in the open right half plane (see, e.g., \cite{LT}, Section 13.4).

\section{Small dimensions}

We begin this section with a small variation on Lemma 2.10 in \cite{BR}, specified to the situation at hand.

\begin{lemma}\label{lem:BRimproved}
Let $H$ be a nonnegative $n\times n$ matrix, and let $A=\rho(H)I-H$. Assume that the zero eigenvalue of $A$ is algebraically simple. Let $v, w$ be nonnegative vectors in $\mathbb{R}^n$. Suppose the vectors $v$ and $w$ satisfy NZP. Then any real eigenvalue $\mu$ of $B(t)=A+tvw^\top$, where $t>0$, satisfies $\mu >0$. 
\end{lemma}

\begin{proof}
From Lemma 2.10 in \cite{BR}, which does not require the condition NZP, we have that any real eigenvalue $\mu$ of $B(t)$ is nonnegative. So it remains to show that the extra condition NZP implies that $\mu >0$. To see this, assume that for some $B(t)$ has eigenvalue zero for some $t>0$, and let $x\not=0$ be such that $B(t)x=0$. Then $0=z_l^\top B(t)x=z_l^\top (A+tvw^\top)x=tz_l^\top vw^\top x$. By NZP $z^\top v\not= 0$, and since also $t>0$, we must have $w^\top x=0$. But then $0=B(t)x=Ax+tvw^\top x=Ax$. Since the zero eigenvalue of $A$ is algebraically simple it follows that $x$ is a nonzero multiple of $z_r$. But then $w^\top x=0$ impies that $w^\top z_r=0$, which contradicts NZP. So $B(t)$ is invertible for all positive $t$. 
\end{proof}

Observe that if $H$ is irreducible, then as we observed before, the condition NZP is automatically satisfied and the zero eigenvalue of $A$ is algebraically simple in that case.

\begin{proposition}
Let $n=2$ or $n=3$.
Let $H$ be an irreducible nonnegative $n\times n$ matrix, let $A=\rho(H)I-H$, and let $v$ and $w$ be nonnegative vectors in $\mathbb{R}^n$, and, in the case that $n=3$, assume $w^\top  v\not= 0$.
Then there is a $t_1>0$ such that for $t>t_1$ the eigenvalues of $B(t)=A+tvw^\top $ are in the open right half plane.
\end{proposition}

For $n=2$ it was already shown in \cite{BR}, Theorem 2.7, part (i) that both eigenvalues will be in the open right half plane for all $t>0$. However, for completeness we present a proof here as well, which also provides more detail about the behaviour of the eigenvalues for large values of $t$.

It will be shown in a later example that the condition $w^\top v\not= 0$ is necessary when $n=3$.

\begin{proof}
\subsubsection*{The case $n=2$.} 
For $n=2$ it will be shown that the eigenvalues of $B(t)$ are in the open right half plane for all $t>0$.
Indeed, as $H$ is irreducible, $A$ has two different eigenvalues, $0$ and a positive number $\mu$.
Hence $\det A =0$ and ${\rm trace\,}A=\mu$.  Then from \cite{RW2}, Proposition 2 (see also \cite{RW1}, Proposition 2.2) we have that the eigenvalues of $B(t)$ which are not eigenvalues of $A$ are the solutions of $w^\top  (\lambda I_2-A)^{-1}v=\frac{1}{t}$. Multiplying left and right with the characteristic polynomial $p_A(\lambda)$ of $A$, one sees that this is equivalent to $\lambda$ being a solution of
\begin{equation}\label{degree21}
\lambda w^\top  v -w^\top ({\rm adj\, }A)v=\frac{1}{t} (\lambda^2-\mu\lambda),
\end{equation}
where ${\rm adj\, }A$ is the adjugate matrix of $A$. In turn, this is equivalent to $\lambda$ being a solution of
\begin{equation}\label{degree2}
\lambda^2-\lambda(\mu+tw^\top  v)+t w^\top ({\rm adj\, }A)v=0.
\end{equation}
If the solutions of this equation are both real, then they have to be positive, by Lemma \ref{lem:BRimproved}. If they are non-real, then the real part of the two solutions $\lambda_{1,2}$ is equal to 
${\rm Re\, }(\lambda_{1,2})=\tfrac{1}{2}(\mu+tw^\top  v) >0$, and hence the eigenvalues lie in the open right half plane. Note that this does not require the condition $w^\top v\not= 0$. 

In fact, when $w^\top v\not=0$, we can be more precise about the behaviour of the eigenvalues for $t\to\infty$.  Either solving for the eigenvalues explicitly from \eqref{degree2} and \eqref{degree21}, or by using \cite{RW2}, Theorem 17 (ii) and (iii), for large $t$ (see also Remark 18 there), we see that one of the eigenvalues of $B(t)$ will go to infinity along the real line and this eigenvalue is equal to $tw^\top  v +O(1)$ .
The other eigenvalue is approximately equal to
$\zeta=\frac{w^\top  {\rm adj\, }(A)v}{w^\top  v}$.
Because $w^\top v\not= 0$ and ${\rm adj\, }(A)$ is a nonnegative matrix, $\zeta>0$. By  \cite{RW2},Theorem 17 ( (v) for large values of $t$ this second eigenvalue is equal to  $\zeta +\frac{r}{t}+O(t^{-2})$, where $r=| w^\top  (\zeta-A)^{-2}v |$.

In the case that $w^\top v=0$ we can also be more precise about the behaviour of the eigenvalues for $t\to\infty$. Solving explicitly for the eigenvalues we see that the eigenvalues will go to infinity along the line ${\rm Re} \lambda_{1,2} =\tfrac{1}{2}\mu$ as $t\to \infty$.

\subsubsection*{The case $n=3.$} In this case we shall show that for large values of $t$ the eigenvalues of $B(t)$ are eventually in the open right half plane under the condition $w^\top v\not= 0$. There are two cases to consider: the first is that for large values of $t$ all eigenvalues are real. This is the easy case, as by Lemma \ref{lem:BRimproved} for large values of $t$ the eigenvalues of $B(t)$ have to be positive. 

In the second case, the matrix $B(t)$ has one real eigenvalue, again positive by  Lemma \ref{lem:BRimproved}, and a pair of complex eigenvalues. In fact, the real eigenvalue must go to infinity along the positive real axis as $tw^\top  v+O(1)$ according to \cite{RW2}, Theorem 17 (ii) and (iii). The complex eigenvalues then have to approximate the two roots of the polynomial $p_{vw}(\lambda)$. So it remains to prove that the roots of $p_{vw}(\lambda)$ are in the open right half plane.

Since $A$ is a singular $M$-matrix and zero is a simple eigenvalue of $A$ by the irreducibility of the nonnegative matrix $H$, the characteristic polynomial of $A$ is of the form $p_A(\lambda)=\lambda^3 +p_2\lambda^2+p_1\lambda$, with $p_2=-{\rm trace\, }A <0$ and $p_1\not=0$. Then, by direct computation, or from \cite{RW2}, Lemma 16,
$$
p_{vw}(\lambda)=\lambda^2 w^\top v + \lambda (p_2 w^\top v +w^\top Av)+
(p_1w^\top v+p_2w^\top Av+ w^\top A^2v).
$$
The roots of $p_{vw}$ are given by
$$
\lambda_{1,2}=\frac{-(p_2 w^\top v +w^\top Av) \pm \sqrt{D}}{2w^\top v},
$$
where $D$ is the discriminant. Obviously, this depends on the sign of $D$. 
The case where the two roots are real has already been solved.
So, we may assume $D<0$. Then the real part of the two roots $\lambda_{1,2}$ is given by
$$
{\rm Re\, } \lambda_{1,2}= \frac{-(p_2 w^\top v +w^\top Av)}{2w^\top v}.
$$
By assumption $w^\top  v >0$, so the sign of the real part of $\lambda_{1,2}$ is equal to the sign of $-(p_2w^\top  v+w^\top  Av)=-w^\top  (p_2 I+A) v$. Now $p_2=-{\rm trace\, }A=-(a_{11}+a_{22}+a_{33})$ and so
$$
p_2I+A =\begin{bmatrix}
-(a_{22}+a_{33}) & a_{12} & a_{13} \\
a_{21} & -(a_{11}+a_{33}) & a_{23}\\
a_{31} & a_{32} & -(a_{11}+a_{22})
\end{bmatrix}.
$$
As $A$ is an irreducible $M$-matrix, its off-diagonal entries are nonpositive, and its diagonal entries are positive, by \cite{BP}, Chapter 6, Theorem 4.16, part 4. Since $w$ and $v$ are nonnegative vectors with $w^\top  v\not=0$, the product $w^\top  (p_2I+A)v$ is negative, as at least for one $i=1,2,3$ we will have $v_i>0$ and $w_i>0$.
It follows that the real part of the roots $\lambda_{1,2}$ of $p_{vw}$ is positive, and hence $\lambda_{1,2}$ are in the open right half plane. Consequently, all three eigenvalues of $B(t)$ are in the open right half plane for $t>0$ large enough.
\end{proof}

For $n=2$ the eigenvalues of $B(t)$ are in the right half plane for all $t>0$. 
The next example shows that for $n=3$ it is not true that the eigenvalues of $B(t)$ are in the open right half plane for all $t>0$.

\begin{example}{\rm
Let $H=\begin{bmatrix} 0.1 & 1 & 0 \\ 0 & 0.1 & 1\\ 10^{-4} & 0 & 0.1\end{bmatrix}$. Then $\rho(H)=0.1464$. Let $v=\begin{bmatrix} 0.6 & 0.1 & 0.3 \end{bmatrix}^\top $ and $w=\begin{bmatrix} 0.5 & 1 &1\end{bmatrix}^\top $. Then the eigenvalues of $B(0.1)$ are $ 0.2661$ and $-0.0284\pm 0.2495i$.

The eigenvalues of $B(t)$ are plotted as functions of $t$ in Figure 2 below.
\begin{figure}[H]
\centering
\includegraphics[height=5.5cm]{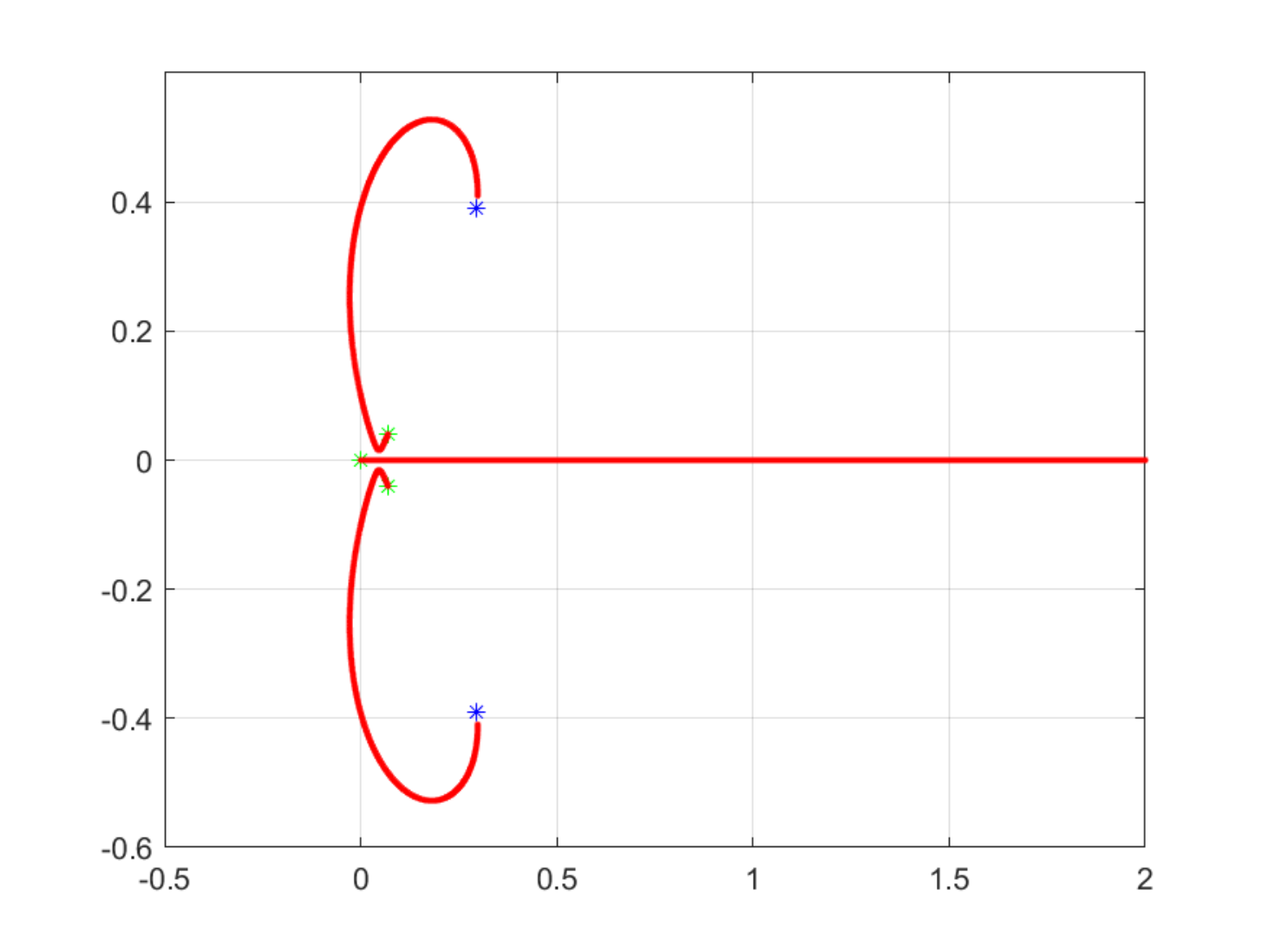}
\caption{Eigenvalue curves of $B(t)$ plotted as functions of $t$. The green stars indicate the eigenvalues of $A=\rho(H)I_3-H$, and the blue stars are the roots of $p_{vw}$.}
\end{figure}
}
\end{example}

The next example shows that the condition $w^\top v\not=0$ cannot be missed when $n=3$.

\begin{example}{\rm 
Let $H=\begin{bmatrix} 0 & 1 & 0 \\ 0 & 0 & 1 \\ 1 & 0 & 0 \end{bmatrix}$. Then $H$ is nonnegative, irreducible and $\rho(H)=1$. So $A=\begin{bmatrix} 1 & -1 & 0 \\ 0 & 1 & -1 \\ -1 & 0 & 1\end{bmatrix}$. Take $v=\begin{bmatrix} 1 \\ 0 \\ 0 \end{bmatrix}$ and $w=\begin{bmatrix} 0 \\ 6 \\ 1 \end{bmatrix}$. Then $w^\top v=0$, $w^\top Av=-1$ and $w^\top A^2v=4$. Put $B(t)=A+tvw^\top =\begin{bmatrix} 1 & -1+6t & t \\ 0 & 1 & -1 \\ -1 & 0 & 1\end{bmatrix}$. One checks that the characteristic polynomial of $B(t)$ is equal to
$$
p_{B(t)}(\lambda)=\lambda^3-3\lambda^2+3\lambda +t(\lambda -7).
$$
Obviously, as $t\to\infty$ one of the roots will go to $7$. 
By the proof of \cite{RW2}, Theorem 17 (iii) there are two eigenvalues of $B(t)$ going to infinity, with a Puiseux series expansion given by
$\sqrt{t}\sqrt{w^\top Av} +\tfrac{1}{2}\cdot \tfrac{w^\top A^2v}{w^\top Av}+O(\tfrac{1}{\sqrt{t}})$. Inserting the values $w^\top Av=-1$ and $w^\top A^2v=4$, this becomes
$\pm\sqrt{t} i -2+O(\tfrac{1}{\sqrt{t}})$. We see that as $t\to\infty$ there are two eigenvalues of $B(t)$ which approximate the line ${\rm Re\, }(\lambda)=-2$, and hence these eigenvalues are in the open left half plane for all $t$ large enough.  

The eigenvalues of $B(t)$ are plotted as functions of $t$ in Figure 3 below.
\begin{figure}[H]
\centering
\includegraphics[height=5.5cm]{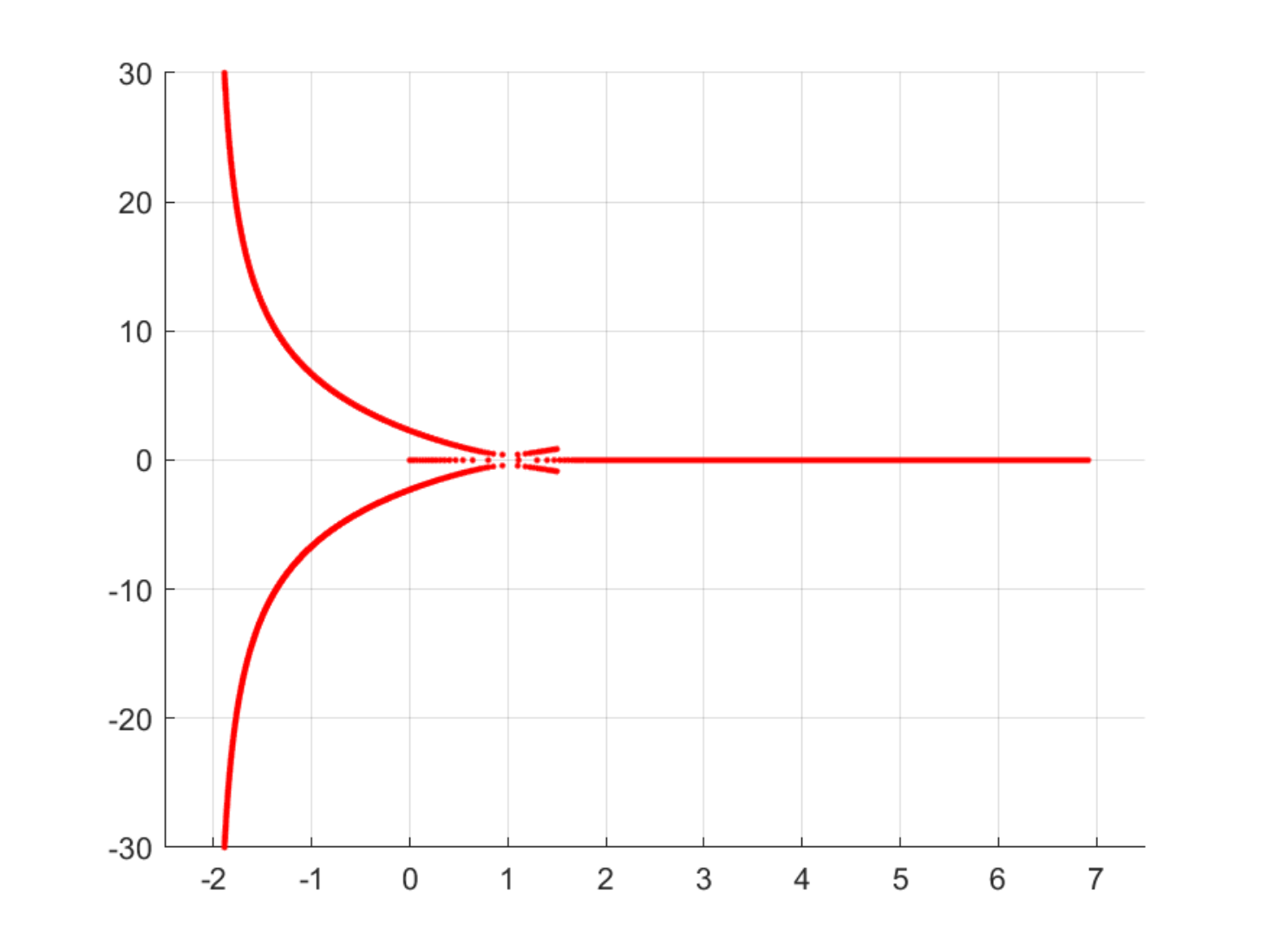}
\caption{Eigenvalue curves of $B(t)$ plotted as functions of $t$.}
\end{figure}

Taking this example one step further, let $A$ and $v$ be as above, but now take $w=\begin{bmatrix}0 \\ 1 \\ 0 \end{bmatrix}$. So $B(t)=\begin{bmatrix} 1 & t-1 & 0 \\ 0 & 1 & -1 \\ -1 & 0 & 1\end{bmatrix}$.
Then $w^\top v=0, w^\top Av=0$ and $w^\top A^2v=1$, and, conforming to \cite{RW2}, Theorem 17 (iii), there are three eigenvalues going to infinity as $t\to\infty$. In fact, since the characteristic polynomial of $B(t)$ is given by $(1-\lambda)^3 +(t-1)$, for $t>1$ these eigenvalues are given by $\lambda_j=1+\sqrt[3]{t-1}e^{2\pi ij/3}$ for $j=0,1,2$. Note that again two of them are in the open left half plane.
}
\end{example}

{\bf Acknowledgement.} The authors would like to express their gratitude to the referee, whose comments led to significant improvements in the presentation of the paper.

\end{document}